\documentclass[11pt]{article}
\usepackage{latexsym,bm}
\usepackage{mathrsfs}
\usepackage{amsmath,amsthm}
\usepackage{graphicx}
\usepackage{amssymb}
\usepackage{CJK}
\usepackage{epstopdf}
\usepackage{color}
\usepackage{cite}
\usepackage{array}
\usepackage{tikz}
\usepackage {indentfirst}

\usetikzlibrary{positioning} 
\usepackage{xcolor}

\theoremstyle{plain}
\newtheorem{thm}{Theorem}[section]
\newtheorem{cor}[thm]{Corollary}
\newtheorem{lem}[thm]{Lemma}
\newtheorem{prop}[thm]{Proposition}
\theoremstyle{definition}

\theoremstyle{plain}

\theoremstyle{problem}

\theoremstyle{plain}

\theoremstyle{plain}

\theoremstyle{plain}

\DeclareGraphicsExtensions{.eps,.eps.gz}
\usepackage[top=2.5cm,bottom=2.5cm,left=2.8cm,right=2.2cm]{geometry}
\normalsize \rm
\allowdisplaybreaks[4]
\makeatletter
\@addtoreset{equation}{section}
\makeatother

\begin{document}

\begin{center}
{\huge On spanning tree edge dependences of graphs} \\[18pt]
{\Large Yujun Yang*\footnotetext{*Corresponding author at E-mail address: yangyj@yahoo.com}, Can Xu, }\\[6pt]
{ \footnotesize School of Mathematics and Information Science, Yantai University, Yantai, 264005 P.R. China}
\end{center}
\vspace{1mm}
\begin{abstract}
Let $\tau(G)$ and $\tau_G(e)$ be the number of spanning trees of a connected graph $G$ and the number of spanning trees of $G$ containing edge $e$. The ratio $d_{G}(e)=\tau_{G}(e)/\tau(G)$ is called the spanning tree edge density of $e$, or simply density of $e$. The maximum density $\mbox{dep}(G)=\max\limits_{e\in E(G)}d_{G}(e)$ is called the spanning tree edge dependence of $G$, or simply dependence of $G$. Given a rational number $p/q\in (0,1)$, if there exists a graph $G$ and an edge $e\in E(G)$ such that $d_{G}(e)=p/q$, then we say the density $p/q$ is constructible. More specially, if there exists a graph $G$ such that $\mbox{dep}(G)=p/q$, then we say the dependence $p/q$ is constructible. In 2002, Ferrara, Gould, and Suffel raised the open problem of which rational densities and dependences are constructible. In 2016, Kahl provided constructions that show all rational densities and dependences are constructible. Moreover, He showed that all rational densities are constructible even if $G$ is restricted to bipartite graphs or planar graphs. He thus conjectured that all rational dependences are also constructible even if $G$ is restricted to bipartite graphs (Conjecture 1), or planar graphs (Conjecture 2). In this paper, by combinatorial and electric network approach, firstly, we show that all rational dependences are constructible via bipartite graphs, which confirms the first conjecture of Kahl. Secondly, we show that all rational dependences are constructible for planar multigraphs, which confirms Kahl's second conjecture for planar multigraphs. However, for (simple) planar graphs, we disprove the second conjecture of Kahl by showing that the dependence of any planar graph is larger than $\frac{1}{3}$. On the other hand, we construct a family of planar graphs that show all rational dependences $p/q>\frac{1}{2}$ are constructible via planar graphs.
\vspace{1mm}

\end{abstract}

\section{Introduction}
Let $G$ be a connected graph. A spanning tree of $G$ is a spanning subgraph that is a tree. The number of spanning trees, denoted by $\tau(G)$, of a graph $G$ is the total number of distinct spanning subgraphs of $G$ that are trees. The number of spanning trees is an important structural graph invariant, which has significant application in network theory, physics, chemistry and engineering. Due to this reason, it has been widely studied for many years. It is well known that Kirchhoff's matrix tree theorem has established a formula for computing the number of spanning trees for a general graph \cite{kir}. Since the pioneering work of Cayley \cite{cay} who first determined the number of spanning trees of complete graphs, the number of spanning trees has been computed for various interesting families of graphs. For example, the number of spanning trees have been computed for complete bipartite graphs \cite{sco,sbe}, complete multipartite graphs \cite{aus,lew}, cubic cycle $C_N^3$ and the quadruple cycle $C_N^4$ \cite{yta}, graphs formed from a complete graph by deleting branches forming disjoint $K$-partite subgraphs \cite{neil}, multi-star related graphs \cite{nr}, $K_n$-complements of quasi-threshold graphs \cite{nr1}, circulant graphs\cite{zyg,lpw,clz,gyy,gyz,lcry}, $K_n$-complements of asteroidal graphs \cite{npp}, graphs with rotational symmetry \cite{yz},  $K_n^m\pm G$ graphs \cite{nr2},  irregular line graphs \cite{yan} and line graphs \cite{dy,gj}, self-similar fractal models \cite{my}, 2-separable networks \cite{ly}, a type of generalized Farey graphs \cite{zyan}, nearly complete bipartite graphs \cite{gd}, Bruhat graph of the symmetric group\cite{moy}, and so on.

In this paper, we concentrate on the number of spanning trees containing a specific edge. For $e\in E(G)$, let $\tau_{G}(e)$ denote the number of spanning trees of $G$ which contain $e$. The ratio $d_{G}(e)=\tau_{G}(e)/\tau(G)$ is called the \textit{spanning tree edge density} \cite{fgs} of $e$, or simply density of $e$. The \textit{spanning tree edge dependence} \cite{fgs} of $G$, or simply dependence of $G$, denoted by $\mbox{dep}(G)$, is defined as $\mbox{dep}(G)=\max_{e\in G}d_{G}(e)$. The spanning tree edge density and dependence have important applications in networks. For instance, if we construct a network $\mathcal{N}$ by replacing each edge of $G$ with a unit resistor, then for any edge $e=uv$, the density of $e$ is equal to the resistance distance between $u$ and $v$ in $\mathcal{N}$. Here, the \textit{resistance distance} between any two vertices $i$ and $j$ in a connected graph $G$, denoted by $\Omega(i,j)$, is defined as the net effective resistance between them in $\mathcal{N}$\cite{kr}.

It is obvious that $0<d_G(e)\leq 1$ and $d_G(e)=1$ if and only if $e$ is a cut edge. Let $p/q$ be a positive rational number, $p<q$. If there exists a graph $G$ with edge
$e\in E(G)$ such that $d_G(e) = p/q$, we say that the spanning tree edge density $p/q$ is constructible. Similarly, if there exists a
graph $G$ with $\mbox{dep}(G)=p/q$, we say that the spanning tree edge dependence $p/q$ is constructible.

In 2002, Ferrara, Gould and Suffel \cite{fgs} proposed the following realizability open problems.

\textbf{Problems.} Which rational spanning tree edge densities are constructible? More specifically, which rational spanning tree edge dependences are constructible?

In 2016, Nathan Kahl \cite{nka} solved the problems by ingenious construction of graph families. By constructing necklace graphs with complete graphs (the definition is given in Section 2), he proved that all rational spanning tree edge densities and dependences are constructible, even if $G$ is restricted to claw-free graphs. In addition, by constructing necklace graphs with complete bipartite graphs, he proved that all rational spanning tree edge densities are constructible even if $G$ is restricted to bipartite graphs. By constructing general theta graphs (the definition is given in Section 3), he proved that all rational spanning tree edge densities are constructible even if $G$ is restricted to planar graphs.

Undoubtedly, it is usually very challenging to determine which edge has the maximum density in a graph. Due to this reason, it remains unknown whether all rational spanning tree edge dependences are constructible if $G$ is restricted to bipartite graphs or planar graphs, although Kahl has shown that all rational spanning tree edge densities are constructible with these graphs. So, Kahl proposed the following conjectures.

\textbf{Conjecture 1.} Let $p,q$ be positive integers, $p<q$. There exists some function $f(p,q)$ such that, if $G$ is the bipartite construction of Theorem 2.1 (see Section 2), then $t_i\geq f(p,q)$ for all $2\leq i \leq n$ implies that $\mbox{dep}(G)=d_G(e_1)$.

\textbf{Conjecture 2.} Let $p,q$ be positive integers, $p<q$. There exists a planar graph $G$ such that $\mbox{dep}(G)=p/q$.

In the present paper, firstly, we do find such a function $f(p,q)$ that, if $G$ is the bipartite construction of Theorem \ref{th1}, then $t_i\geq f(p,q)$ for all $2\leq i \leq n$ implies that $\mbox{dep}(G)=d_G(e_1)$, which confirms Conjecture 1. Secondly, by electrical network approach, we show that Conjecture 2 is true for planar multigraphs. However, for (simple) planar graphs, we disprove the second conjecture of Kahl by showing that the dependence of any planar graph is larger than $\frac{1}{3}$. Hence all rational dependences $p/q$ for $p/q\leq \frac{1}{3}$ are not constructible. On the other hand, we construct a family of planar graphs that show all rational dependences $p/q>\frac{1}{2}$ are constructible via planar graphs.

\section{Constructing spanning tree edge dependences via bipartite graphs}
In this section, we show that Conjecture 1 is true. To this end, we use the same bipartite construction as given in \cite{nka}. So, first of all, we introduce the construction of necklace graphs.

Informally speaking, the necklace graph is obtained from a cycle by replacing each edge of the cycle with a graph. Precisely, let $\{G_{i}\}_{i=1}^{n}$ be a sequence of graphs, and in each graph $G_i$, we choose an edge $e_{i}=u_{i}v_{i}$. Then the \textit{necklace graph} $G=N(G_{1}(e_{1}),\cdots, G_{n}(e_{n}))$ is the graph obtained from $\{G_{i}\}_{i=1}^{n}$ by connecting these $n$ edges in turn according to the sequence of the graph in which they are lied, that is, identify each $v_{i}$ and $u_{i+1}$(the indices are modulo $n$). If there is no confusion, for the sake of simplicity, we write $G=N(G_{1},\cdots, G_{n})$ instead. Let $K_{r,s}$ be the complete bipartite graph with two partite sets having $r$ and $s$ vertices, respectively. Then the bipartite construction used in this section is the necklace graph $G=N(K_{r_1,s_1},K_{r_2,s_2}\cdots, K_{r_n,s_n})$ constructed from complete bipartite graphs. Clearly, if $n$ is even, then $G=N(K_{r_1,s_1},K_{r_2,s_2}\cdots, K_{r_n,s_n})$ is a bipartite graph. For example, the necklace graph $G=N(K_{1,1},K_{2,3},K_{3,3},K_{3,4})$ is shown in figure 1.

Using the bipartite construction, Kahl proved that for any positive rational number $0<p/q<1$, the spanning tree edge density $p/q$ is constructible, as stated in the following theorem.

\begin{thm}\label{th1}\cite{nka}
Let $p,\ q$ be positive integers, $p< q$, let $t_{i}$, $2\leq i\leq n$, be positive integers such that $$\sum\limits_{i=2}\limits^{n}\frac{1}{t_{i}}=\frac{p}{q-p}.$$ Let $r_1=s_1=1$ and, for all $2\leq i\leq n$, let $r_{i}=2t_{i}$ and $s_{i}=2t_{i}-1$. Then in $G=N(K_{r_1,s_1}, K_{r_{2},s_{2}},\ \cdots,\ K_{r_{n},s_{n}})$, we have $d_{G}(e_{1})=p/q$.
\end{thm}

\begin{figure*}[!ht]
 \centering
 $$\includegraphics[width=5.7cm,height=4.3cm]{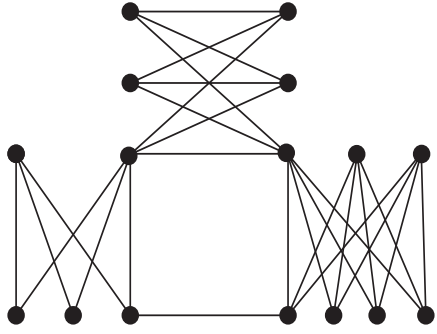}$$\\
 \caption{The necklace graph $G=N(K_{1,1},\ K_{2,3},\ K_{3,3},\ K_{3,4})$.}     \label{Fig 1}
\end{figure*}

To show that the spanning tree edge dependence $p/q$ is constructible in the bipartite construction as given in Theorem \ref{th1}, it is needed to compute densities of all the edges in $G=N(K_{1,1},K_{r_2,s_2}\cdots, K_{r_n,s_n})$. For convenience, we divide all the edges into three groups. The first group consists of all the edges $e_{i}$ ($i=1,\ 2,\cdots,\ n$), these edges are called \textit{key edges} of $G$, and the vertices $u_{i}$ and $v_{i}$ are the \textit{key vertices} of $G$. The second group consists of those non-key edges which are adjacent to key edges, and we call edges in the second group \textit{type 1 edges}. The third group consists of the remaining edges, and these edges are called \textit{type 2 edges}.

Actually, to single out the edge in $G=N(K_{1,1},K_{r_2,s_2}\cdots, K_{r_n,s_n})$ with the maximum spanning tree edge density, it suffices to compute $\tau_G(xy)$ for all the edges $xy\in E(G)$. To compute $\tau_G(xy)$, the following lemma plays an essential role. Note that the \textit{spanning thicket}\cite{bbo} (or \textit{spanning bitree}) of a graph is a spanning forest with exactly two components. If vertices $u$ and $v$ lie in different components of a spanning thicket, then we say the spanning thicket \textit{separates} $u$ and $v$.


\begin{lem}\label{lemma3}\cite{nka}
Let $G=N(G_{1},\cdots,G_{n})$ be a necklace graph. Then
\begin{center}
 $\tau(G)=\prod\limits_{i=1}\limits^{n}\tau(G_{i})\sum\limits_{i=1}\limits^{n}d_{G_{i}}(u_{i}v_{i})$
\end{center}
and, for any $xy\in E(G_{k})$, $1\leq k\leq n$,
\begin{center}
 $\tau_{G}(xy)=\prod\limits_{i=1}\limits^{n}\tau(G_{i})\left[\frac{b_{G_{k}}(xy;\ u_{k},\
  v_{k})}{\tau(G_{k})}+d_{G_{k}}(xy)\sum\limits_{i\neq k}d_{G_{i}}(u_{i}v_{i})\right]$
  \end{center}
 where $b_{G_{k}}(xy;\ u_{k},v_{k})$ is the number of spanning thickets of $G_{k}$ separating $u_{k}$, $v_{k}$ and contain the edge $xy$.
   \end{lem}

In order to compute $\tau_{G}(xy)$ in $G=N(K_{1,1},K_{r_2,s_2}\cdots, K_{r_n,s_n})$ by Lemma \ref{lemma3}, we still need to two results of Ge and Dong in \cite{gd} which gives the number of spanning trees which contain two types of subgraphs of a complete bipartite graph $G=K_{r,s}$. For a graph $G$ and a subgraph $H$ of $G$, we use $\tau_{G}(H)$ to denote the number of spanning trees of $G$ containing $H$. The first result concerns the number of spanning tree of $K_{r,s}$ that contains a given subgraph tree $T$ of $K_{r,s}$.

\begin{lem}\label{lemma1}\cite{gd}
Let $T$ be any tree which is a subgraph of $K_{r,s}$. Then

\begin{equation}
\tau_{K_{r,s}}(T)=(ms+nr-mn)r^{s-n-1}s^{r-m-1},
\end{equation}
where $m=|V(T)\bigcap X|$, $n=|V(T)\bigcap Y|$, and $(X,Y)$ is the bipartition of $K_{r,s}$, with $|X|=r$, and $|Y|=s$.
\end{lem}

The second result enumerates the number of spanning trees of $K_{r,s}$ containing a given matching of $K_{r,s}$.

\begin{lem}\label{lemma2}\cite{gd}
For any matching $M$ of size $l$ in $K_{r,s}$, we have
\begin{equation}
 \tau_{K_{r,s}}(M)=(r+s)^{l-1}(r+s-l)r^{s-l-1}s^{r-l-1}.
\end{equation}
   \end{lem}

Before computing densities in $G=N(K_{1,1},\ K_{r_{2},s_{2}},\ \ldots,\ K_{r_{n},s_{n}})$, the following lemma is needed.
\begin{lem}\label{trans}\cite{fgs,kr}
 Let $G$ be an edge-transitive graph with $n$ vertices and $m$ edges, then for any edge $e$ in $G$, we have
 \begin{equation}
 d_{G}(e)=\frac{n-1}{m}.
 \end{equation}
   \end{lem}

Now we are ready to compute the number of spanning trees containing a non-key edge $e$ in $G=N(K_{1,1},\ K_{r_{2},s_{2}},\ \ldots,\ K_{r_{n},s_{n}})$. For $2\leq k\leq n$, suppose that the two partitions of $K_{r_k,s_k}$ are $X_k$ and $Y_k$, $|X_k|=r_k$ and $|Y_k|=s_k$, $u_k\in X_k$ and $v_k\in Y_k$.

\begin{thm}\label{tm1}
Let $G=N(K_{1,1},\ K_{r_{2},s_{2}},\ \ldots,\ K_{r_{n},s_{n}})$. If $xy\in E(K_{r_{k},s_{k}})$ is a type 1 edge of $G$, then
\begin{equation}\label{e21}
\tau_{G}(xy)= \left\{ \begin{array}{rr}\prod \limits_{i=1}\limits^{n}r_{i}^{s_{i}-1}s_{i}^{r_{i}-1}\left[\frac{r_{k}+s_{k}-1}{r_{k}s_{k}}\sum \limits_{i=1}\limits^{n}\frac{r_{i}+s_{i}-1}{r_{i}s_{i}}-\frac{(r_{k}-1)^{2}}{r_{k}^{2}s_{k}^{2}}\right], & \{x,y\}\cap \{u_k,v_k\}=\{u_k\},\\
  \prod \limits_{i=1}\limits^{n}r_{i}^{s_{i}-1}s_{i}^{r_{i}-1}\left[\frac{r_{k}+s_{k}-1}{r_{k}s_{k}}\sum \limits_{i=1}\limits^{n}\frac{r_{i}+s_{i}-1}{r_{i}s_{i}}-\frac{(s_{k}-1)^{2}}{r_{k}^{2}s_{k}^{2}}\right], & \{x,y\}\cap \{u_k,v_k\}=\{v_k\}.
  \end{array}\right.
\end{equation}
If $xy\in E(K_{r_{k},s_{k}})$ is a type 2 edge of $G$, then
 \begin{equation}\label{e22}
\tau_{G}(xy)=\prod \limits_{i=1}\limits^{n}r_{i}^{s_{i}-1}s_{i}^{r_{i}-1}\left[\frac{r_{k}+s_{k}-1}{r_{k}s_{k}}\sum \limits_{i=1}\limits^{n}\frac{r_{i}+s_{i}-1}{r_{i}s_{i}}-\frac{1}{r_{k}^{2}s_{k}^{2}}\right].
  \end{equation}
\end{thm}

\begin{proof}
 We first consider the case that $xy\in E(K_{r_{k},s_{k}})$ is a type 1 edge of $G$. By Lemma~\ref{lemma3}, we know that
 \begin{equation}\label{e23}
 \tau_{G}(xy)=\prod \limits_{i=1}\limits^{n}\tau(K_{r_{i},s_{i}})\left[\frac{b_{K_{r_{k},s_{k}}}(xy;\ u_{k},v_{k})}{\tau(K_{r_{k},s_{k}})}+d_{K_{r_{k},s_{k}}}(xy)\sum\limits_{i\neq k}d_{K_{r_{i},s_{i}}}(u_{i}v_{i})\right].
 \end{equation}
 For the complete bipartite graph $K_{r,s}$, it has been obtained in \cite{fs} that
 \begin{equation}\label{e24}
 \tau(K_{r,s})=r^{s-1}s^{r-1}.
 \end{equation}
 In addition, since $K_{r,s}$ is edge-transitive, the spanning tree edge density of any edge $e$ of $K_{r,s}$ is
 \begin{equation}\label{e25}
 d_{K_{r,s}}(e)=\frac{r+s-1}{rs}.
 \end{equation}
 Now consider $b_{K_{r_{k},s_{k}}}(xy;\ u_{k},v_{k})$. Since $xy$ is a type 1 edge, $xy$ is adjacent to the key edge $u_kv_k$. Suppose that $y=u_k$ is the common end-vertex of $xy$ and $u_kv_k$. Then the path $P=xyv_k$ is tree which is a subgraph of $K_{r_{k},s_{k}}$. For each spanning tree $T$ of $K_{r_{k},s_{k}}$ containing $P$, if we delete $u_kv_k$ from $T$, then we could obtain a spanning thicket $B$ of $K_{r_{k},s_{k}}$ separating $u_k$ and $v_k$ that contains the edge $xy$. Conversely, for each spanning thicket $B$ of $K_{r_{k},s_{k}}$ separating $u_k$ and $v_k$ that contains the edge $xy$, if we add the edge $u_kv_k$ to $B$, then we obtain a spanning tree $T$ of $K_{r_{k},s_{k}}$ containing $P$. It establishes a one to one correspondence between the set of spanning trees of $K_{r_{k},s_{k}}$ containing $P$ and the set of spanning thickets of $K_{r_{k},s_{k}}$ separating $u_k$ and $v_k$ that contains the edge $xy$. Hence
 \begin{equation}\label{e26}
 b_{K_{r_{k},s_{k}}}(xy;\ u_{k},v_{k})=\tau_{K_{r_{k},s_{k}}}(P).
 \end{equation}
Then by Lemma \ref{lemma1}, we have
\begin{equation}\label{e27}
\tau_{K_{r_k,s_k}}(P)=(s_k+2r_k-2)r_k^{s_k-3}s_k^{r_k-2},
\end{equation}
Substituting results of Eqs. (\ref{e24}), (\ref{e25}), and (\ref{e27}) into Eq. (\ref{e23}), we obtain
\begin{align*}
\tau_{G}(xy)&=\prod \limits_{i=1}\limits^{n}\tau(K_{r_{i},s_{i}})\left[\frac{b_{K_{r_{k},s_{k}}}(xy;\ u_{k},v_{k})}{\tau(K_{r_{k},s_{k}})}+d_{K_{r_{k},s_{k}}}(xy)\sum\limits_{i\neq k}d_{K_{r_{i},s_{i}}}(u_{i}v_{i})\right]\\
&=\prod \limits_{i=1}\limits^{n}r_{i}^{s_{i}-1}s_{i}^{r_{i}-1}\left[\frac{(s_k+2r_k-2)r_k^{s_k-3}s_k^{r_k-2}}{r_k^{s_k-1}s_k^{r_k-1}}+\frac{r_k+s_k-1}{r_ks_k}\sum\limits_{i\neq k}\frac{r_i+s_i-1}{r_is_i}\right]\\
&=\prod \limits_{i=1}\limits^{n}r_{i}^{s_{i}-1}s_{i}^{r_{i}-1}\left[\frac{s_k+2r_k-2}{r_k^2s_k}+\frac{r_k+s_k-1}{r_ks_k}\left(\sum\limits_{i=1}^n\frac{r_i+s_i-1}{r_is_i}-\frac{r_k+s_k-1}{r_ks_k}\right)\right]\\
&=\prod \limits_{i=1}\limits^{n}r_{i}^{s_{i}-1}s_{i}^{r_{i}-1}\left[\frac{r_k+s_k-1}{r_ks_k}\sum\limits_{i=1}^n\frac{r_i+s_i-1}{r_is_i}+\frac{s_k+2r_k-2}{r_k^2s_k}-\left(\frac{r_k+s_k-1}{r_ks_k}\right)^2\right]\\
&=\prod \limits_{i=1}\limits^{n}r_{i}^{s_{i}-1}s_{i}^{r_{i}-1}\left[\frac{r_k+s_k-1}{r_ks_k}\sum\limits_{i=1}^n\frac{r_i+s_i-1}{r_is_i}-\frac{(r_k-1)^2}{r_k^2s_k^2}\right].
\end{align*}
Hence the first equality in Eq. (\ref{e21}) is obtained. The second equality in Eq. (\ref{e21}) could be obtained in the same way.

Now suppose that $xy\in E(K_{r_{k},s_{k}})$ is a type 2 edge of $G$. Then $xy$ and $u_kv_k$ form a matching $M$ of $K_{r_k,s_k}$, where both $x_{k}$ and $y_{k}$ distinct from key vertices $u_{k},\ v_{k}$. It is not difficult to verify that there exists a one to one correspondence between the set of spanning trees of $K_{r_k,s_k}$ containing $M$ and the set of spanning thickets of $K_{r_k,s_k}$ separating $u_k$ and $v_k$ that contains the edge $xy$. Thus by Lemma \ref{lemma2}, we have
\begin{equation}\label{e28}
 b_{K_{r_{k},s_{k}}}(xy;u_{k},v_{k})=\tau_{K_{r_k,s_k}}(M)=(r_k+s_k)(r_k+s_k-2)r_k^{s_k-3}s_k^{r_k-3}.
\end{equation}
Substituting results of Eqs. (\ref{e24}), (\ref{e25}), and (\ref{e28}) into Eq. (\ref{e23}), we obtain
\begin{align*}
\tau_{G}(xy)&=\prod \limits_{i=1}\limits^{n}r_{i}^{s_{i}-1}s_{i}^{r_{i}-1}\left[\frac{(r_{k}+s_{k})(r_{k}+s_{k}-2)r_{k}^{s_{k}-3}s_{k}^{r_{k}-3}}{r_{k}^{s_{k}-1}s_{k}^{r_{k}-1}}+\frac{r_{k}+s_{k}-1}{r_{k}s_{k}}\sum\limits_{i\neq k}\frac{r_{i}+s_{i}-1}{r_{i}s_{i}}\right]\\
&=\prod \limits_{i=1}\limits^{n}r_{i}^{s_{i}-1}s_{i}^{r_{i}-1}\left[\frac{(r_{k}+s_{k})(r_{k}+s_{k}-2)}{r_{k}^{2}s_{k}^{2}}
+\frac{r_{k}+s_{k}-1}{r_{k}s_{k}}\sum\limits_{i=1}^n\frac{r_{i}+s_{i}-1}{r_{i}s_{i}}-\frac{(r_{k}+s_{k}-1)^{2}}{r_{k}^{2}s_{k}^{2}}\right]\\
&=\prod \limits_{i=1}\limits^{n}r_{i}^{s_{i}-1}s_{i}^{r_{i}-1}\left[\frac{r_{k}+s_{k}-1}{r_{k}s_{k}}\sum \limits_{i=1}\limits^{n}\frac{r_{i}+s_{i}-1}{r_{i}s_{i}}-\frac{1}{r_{k}^{2}s_{k}^{2}}\right].
\end{align*}
Hence Eq. (\ref{e22}) is proved.
\end{proof}

To show that all rational spanning tree dependences are constructible via bipartite graphs, the following result is needed.
 \begin{lem}\label{lemma5}\cite{nka}
Let $G=N(G_{1},\cdots,G_{n})$ be a necklace graph with key edges $e_{i}\in E(G_{i})$, $1\leq i \leq n$. Then $d_{G}(e_{k}) \leq d_{G}(e_{l})$ if and only if $d_{G_{k}}(e_{k})\leq d_{G_{l}}(e_{l})$.
   \end{lem}

Now we are ready to give the main result of this section.

\begin{thm}\label{tm2}
Let $p,\ q$ be positive integers, $p< q$. Let $t_{i}$, $2\leq i\leq n$, be positive integers such that
$$\sum\limits_{i=2}\limits^{n}\frac{1}{t_{i}}=\frac{p}{q-p}\quad \mbox{and} \quad t_{i}\geq\frac{q}{p}.$$ Let $r_{i}=2t_{i}$ and $s_{i}=2t_{i}-1$, for all $2\leq i\leq n$. Then in $G=N(K_{1,1}, K_{r_{2},s_{2}},\ \cdots,\ K_{r_{n},s_{n}})$, we have $\mbox{dep}(G)=d_{G}(e_{1})=p/q$.
\end{thm}

\begin{proof}
 Since $\sum\limits_{i=2}\limits^{n}\frac{1}{t_{i}}=\frac{p}{q-p}$ satisfies the condition in Theorem \ref{th1}, we readily have $d_{G}(e_{1})=p/q$. To get the required result, it suffices to show that if, for all $2\leq i\leq n$, $t_{i}\geq\frac{q}{p}$, then $\mbox{dep}(G)=d_G(e_1)$.

 For any key edge $e_k$ ($2\leq k\leq n$), since $t_k\geq q/p> 1$, and $t_k$ is positive integers, we have $t_k\geq 2$. Then it follows that $r_k=2t_k\geq 4$ and $s_k=2t_k-1\geq 3$. Thus,
 $$d_{K_{r_k,s_k}}(e_k)=\frac{r_k+s_k-1}{r_ks_k}<1=d_{K_{1,1}}(e_1).$$
 Hence by Lemma \ref{lemma5}, we get that for $2\leq k\leq n$, $d_G(e_k)<d_G(e_1)$.

 Now we compare densities of $e_1$ and non-key edges. By Theorem \ref{tm1}, it is easily seen that for each $2\leq k\leq n$, if $x_1y_1\in E(K_{r_{k},s_{k}})$ is a type 1 edge of $G$ and $x_2y_2\in E(K_{r_{k},s_{k}})$ is a type 2 edge of $G$, then
 $$\tau_G(x_1y_1)<\tau_G(x_2y_2).$$
 So, in order to show $e_1$ has maximum density, we only need to show that the density of $e_1$ is larger than that of any type 2 edge.

 By Lemma \ref{lemma3}, it is easily verified that
 \begin{equation}\label{e29}
 \tau_{G}(e_{1})=\prod\limits_{i=1}\limits^{n}\tau(G_{i})\left[\sum\limits_{i\neq k}d_{G_{i}}(u_{i}v_{i})\right]=\prod \limits_{i=1}\limits^{n}r_{i}^{s_{i}-1}s_{i}^{r_{i}-1}\left[\sum\limits_{i=2}^n\frac{r_{i}+s_{i}-1}{r_{i}s_{i}}\right].
\end{equation}
Let $xy\in E(K_{r_{k},s_{k}})$ be a type-2 edge. Then by Theorem \ref{tm1}, we have
\begin{align*}
&\tau_{G}(e_{1})-\tau_{G}(xy)\\
&=\prod \limits_{i=1}\limits^{n}r_{i}^{s_{i}-1}s_{i}^{r_{i}-1}\left[\sum\limits_{i=2}\frac{r_{i}+s_{i}-1}{r_{i}s_{i}}\right]-\prod \limits_{i=1}\limits^{n}r_{i}^{s_{i}-1}s_{i}^{r_{i}-1}\left[\frac{r_{k}+s_{k}-1}{r_{k}s_{k}}\sum \limits_{i=1}\limits^{n}\frac{r_{i}+s_{i}-1}{r_{i}s_{i}}-\frac{1}{r_{k}^{2}s_{k}^{2}}\right]\\
&=\prod \limits_{i=1}\limits^{n}r_{i}^{s_{i}-1}s_{i}^{r_{i}-1}\left[\sum\limits_{i=2}^{n}\frac{r_{i}+s_{i}-1}{r_{i}s_{i}}-\frac{r_{k}+s_{k}-1}{r_{k}s_{k}}\sum \limits_{i=1}\limits^{n}\frac{r_{i}+s_{i}-1}{r_{i}s_{i}}+\frac{1}{r_{k}^{2}s_{k}^{2}}\right]\\
&=\prod \limits_{i=1}\limits^{n}r_{i}^{s_{i}-1}s_{i}^{r_{i}-1}\left[\sum\limits_{i=2}^{n}\frac{r_{i}+s_{i}-1}{r_{i}s_{i}}-\frac{r_{k}+s_{k}-1}{r_{k}s_{k}}\sum \limits_{i=2}\limits^{n}\frac{r_{i}+s_{i}-1}{r_{i}s_{i}}-\frac{r_{k}+s_{k}-1}{r_{k}s_{k}}+\frac{1}{r_{k}^{2}s_{k}^{2}}\right]\\
&=\prod \limits_{i=1}\limits^{n}r_{i}^{s_{i}-1}s_{i}^{r_{i}-1}\left[\sum\limits_{i=2}^{n}\frac{4t_i-2}{2t_i(2t_i-1)}-\frac{4t_k-2}{2t_k(2t_k-1)}\sum \limits_{i=2}\limits^{n}\frac{4t_i-2}{2t_i(2t_i-1)}-\frac{4t_k-2}{2t_k(2t_k-1)}+\frac{1}{r_{k}^{2}s_{k}^{2}}\right]\\
&=\prod\limits_{i=1}\limits^{n}r_{i}^{s_{i}-1}s_{i}^{r_{i}-1}\left[\frac{p}{q-p}(1-\frac{1}{t_{k}})-\frac{1}{t_{k}}+\frac{1}{r_{k}^{2}s_{k}^{2}}\right]\\
&\geq\prod\limits_{i=1}\limits^{n}r_{i}^{s_{i}-1}s_{i}^{r_{i}-1}\left[\frac{p}{q-p}(1-\frac{p}{q})-\frac{p}{q}+\frac{1}{r_{k}^{2}s_{k}^{2}}\right]\\
&=\prod\limits_{i=1}\limits^{n}r_{i}^{s_{i}-1}s_{i}^{r_{i}-1}\cdot\frac{1}{r_{k}^{2}s_{k}^{2}}>0.
\end{align*}

Consequently, $\mbox{dep}(G)=\max_{e\in G}d_{G}(e)=d_{G}(e_{1})=p/q$.
\end{proof}

Since in Theorem \ref{tm2}, we can always insure that $n$ is even, it implies that the necklace graph can be always chosen to be bipartite. Thus the following result is obvious.

\begin{cor}\label{cor}
Let $p$, $q$ be positive integers, $p<q$. Then there exists a bipartite graph $G$, such that $\mbox{dep}(G)=p/q$.
\end{cor}

\section{Constructing spanning tree edge dependences via planar graphs}
\subsection{The case of planar multigraphs}
If the graph have multiple edges, then the graph is called a \textit{multigraph}. In this section, we show that all rational spanning trees edge dependences are constructible via planar multigraphs.

In \cite{nka}, Kahl defined generalized theta graph. A \textit{generalized theta graph} $\Theta(r_1 , r_2,\ldots, r_n )$ is a graph consisting of two distinguished vertices $u, v $ with $n$ disjoint paths between them, of lengths (in edges) of $r_1, r_2,\ldots,r_n$. For example, the generalized theta graph $\Theta(1, 2, 3, 7)$ is shown in figure 2 (left). Using the construction of generalized theta graph, he proved that all rational spanning tree edge densities are constructible via planar graphs.
\begin{figure*}[!ht]

\begin{minipage}[t]{0.55\linewidth}
\centering
\includegraphics[width=5cm,height=5cm]{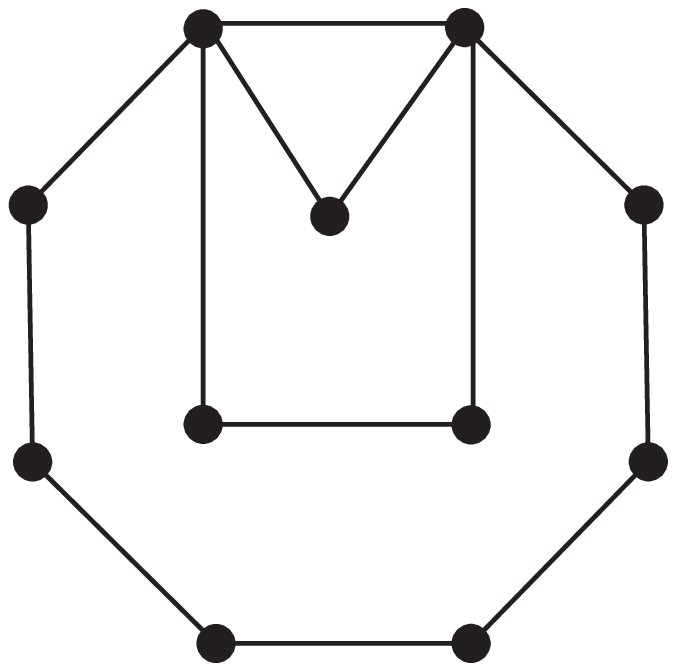}
\end{minipage}
\begin{minipage}[t]{0.5\linewidth}        
\hspace{1mm}
\includegraphics[width=7cm,height=6cm]{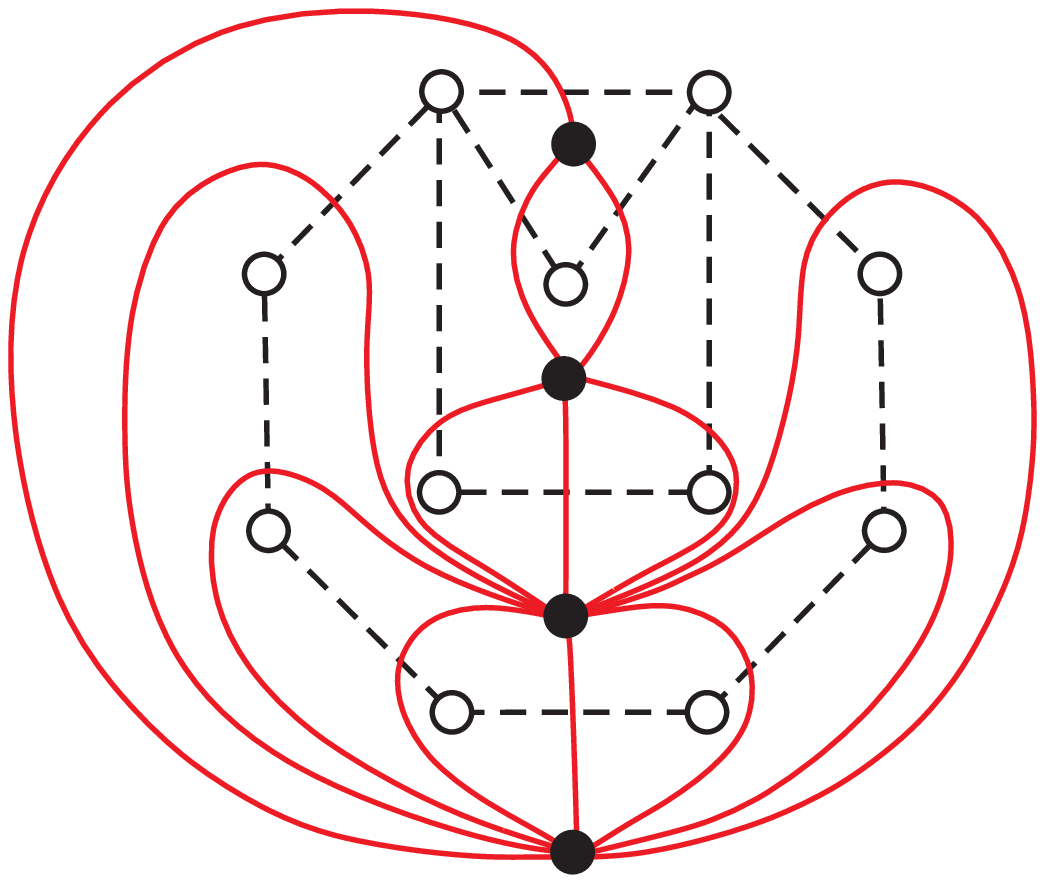}
\end{minipage}
\caption{The generalized theta graph $\Theta(1, 2, 3, 7)$ (left) and its dual graph (right)}
\end{figure*}
\begin{thm}\label{planar}
Let $p$, $q$ be positive integers, $p<q$. Let $G = \Theta(1,r_2,\ldots,r_n)$, with
$$\sum_{i=2}^n\frac{1}{r_i}=\frac{q-p}{p}.$$
Then $d_G(uv)=p/q.$
\end{thm}

In addition, the number of spanning trees of $G=\Theta(r_1 , r_2,\ldots, r_n )$ and the number of spanning trees containing an specific edge $e$ are also determined.

\begin{lem}\cite{nka}
Let $G=\Theta(r_1 , r_2,\ldots, r_n )$ be a generalized theta graph. Then
$$\tau(G)=\prod \limits_{i=1}^{n}r_i\left(\sum_{i=1}^{n}\frac{1}{r_i}\right),$$
and if $e_k$ is an edge of the $k$th path of $G$, then
$$\tau_G(e_k)=\prod \limits_{i=1}^{n}r_i\left(\sum_{i=1}^{n}\frac{1}{r_i}-\frac{1}{r_k}\sum_{i\neq k}\frac{1}{r_i}\right).$$
\end{lem}

From the above lemma, it is easily verified that in $G=\Theta(1, r_2,\ldots, r_n )$, $uv$ is the unique edge with minimum density, which seems to be depressing since the dependence of $G$ has nothing to do with the density of $uv$. Fortunately, we could construct a new graph $G^*$, called the dual graph of $G$, such that $u^*v^*$, the dual edge of $uv$ in $G^*$, has maximum density in $G^*$.

Now we introduce the dual graph of a planar graph. Let $G$ be a planar graph and we draw $G$ in the plane in such a way that no two edges intersect-except at a vertex to which they are both incident. In such a drawing, $G$ is a plane graph. The \textit{dual graph} of the plane graph $G$, denoted by $G^*$, is a plane graph whose vertices correspond to the faces of $G$ and the edges of $G^*$ correspond to edges of $G$ as follows: if $e$ is an edge of $G$ with face $f_1$ on one side and face $f_2$ on the other side, then the end vertices of the dual edge $e^*$ are the vertices that represent the faces $f_1$, $f_2$ of $G$. We can embed $G$ and $G^*$ simultaneously in the plane, such that an edge $e$ of $G$ crosses the corresponding dual edge $e^*$ of $G^*$ exact once and crosses no other edges of $G^*$. For example, the dual graph of the generalized theta graph $\Theta(1, 2, 3, 7)$ is shown in figure 2 (right).

Let $e=uv$ be an edge of $G$ and let $e^*=u^*v^*$ be its dual edge in $G^*$. Thomassen established a nice relation between $\Omega_G(u,v)$ and $\Omega_{G^*}(u^*,v^*)$.
\begin{prop}\label{p1}\cite{tho}
Let $G$ be a planar graph and let $G^*$ be the geometric dual of $G$. Let $e=uv\in E(G)$ and $e^*=(u^*,v^*)\in E(G^*)$ be a pair of dual edges. Then
\begin{equation}
\Omega_G(u,v)+\Omega_{G^*}(u^*,v^*)=1.
\end{equation}
\end{prop}

From Proposition \ref{p1}, we could draw the conclusion that all rational spanning tree edge dependences are constructible via planar multigraphs.

\begin{thm}
Let $p<q$ be positive integers, $p<q$. Let $G = \Theta(1,r_2,\ldots,r_n)$, with
$$\sum_{i=2}^n\frac{1}{r_i}=\frac{p}{q-p}.$$
Let $G^*$ be the dual graph of $G$, and let $e^*=u^*v^*$ be the dual edge of $e=uv$. Then $\mbox{dep}(G^*)=d_{G^*}(u^*v^*)=p/q.$
\end{thm}
\begin{proof}
Since the edge $uv$ is the unique edge with minimum density in $G$, it following from Proposition \ref{p1} that the edge $u^*v^*$ is the unique edge with maximum density in $G^*$. By Theorem \ref{planar}, we get that if $\sum\limits_{i=2}^n\frac{1}{r_i}=\frac{p}{q-p},$ then $d_G(uv)=1-p/q$. Hence, again by Proposition \ref{p1} , we get that $\mbox{dep}(G^*)=d_{G^*}(u^*v^*)=p/q,$ as desired.
\end{proof}

As a straightforward consequence, we get

\begin{cor}\label{cor}
Let $p$, $q$ be positive integers, $p<q$. Then there exists a planar multigraph $G$, such that $\mbox{dep}(G)=p/q$.
\end{cor}

\subsection{The case of (simple) planar graphs}
In this section, we first show that for all $p/q\leq \frac{1}{3}$, spanning tree edge dependences $p/q$ are not constructible via planar graphs. We first introduce the following lemma, which is also known as the famous Foster's first formula in electric network theory \cite{foster}.

\begin{lem}\cite{fgs}\label{foster}
Let $G$ be a connected graph on $n$ vertices. Then
\begin{equation}
\sum_{e\in E(G)}d_G(e)=n-1.
\end{equation}
\end{lem}

By Lemma \ref{foster}, we could get the following result.
\begin{thm}\label{th37}
Let $G$ be a planar graph on $n$ vertices. Then
\begin{equation}
\mbox{dep}(G)>\frac{1}{3}.
\end{equation}
\end{thm}
\begin{proof}
Since $G$ is a planar graph, $|E(G)|\leq 3n-6$. Hence
$$\mbox{dep}(G)\leq \frac{n-1}{3n-6}>\frac{1}{3}.$$
\end{proof}

Theorem \ref{th37} implies that Conjecture 2 fails for $p/q\leq \frac{1}{3}$. That is, all rational dependences $p/q\leq \frac{1}{3}$ are not constructible via planar graphs.

In the rest of this section, we show that for $p/q>\frac{1}{2}$, the spanning tree edge dependence $p/q$ is constructible via planar graphs. To this end, we construct a new family of planar graphs.
 Let $G=N(H_{r_{1}},H_{r_{2}},\cdots, H_{r_{n}})$ be a necklace graph, where $H_{r_{i}}$ is a graph consisting of two key vertices $u_{i}$, $v_{i}$ with one key edge $e_{i}=u_iv_i$ and $r_{i}$ disjoint paths of length 2 connecting $u_i$ and $v_i$. In particular, if $r_{i}=0$, then $H_{0}$ is the graph only consists of the edge $u_{i}v_{i}$. Clearly, the graph $G=N(H_{r_{1}},H_{r_{2}},\cdots, H_{r_{n}})$ is planar. For example, the graph $G=N(H_{0},H_{4},H_{1},H_{3},H_{2},H_{3})$ is shown in Figure 3.
\begin{figure*}[!ht]
 \centering
 $$\includegraphics[width=5.5cm,height=5cm]{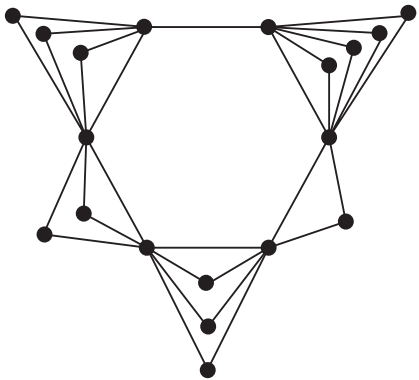}$$\\
 \caption{The planar graph $G=N(H_{0},H_{4},H_{1},H_{3},H_{2},H_{3})$.}   \label{Fig 3}
\end{figure*}

Now we show that for any rational number $0<p/q<1$, the density $p/q$ is constructible from $G=N(H_{0},H_{r_{2}},\cdots, H_{r_{n}})$.

\begin{thm}\label{d{G}(e)}
Let $p,\ q$ be positive integers, $p<q$. Let $G=N(H_{0},H_{r_{2}},\cdots, H_{r_{n}})$ such that $$\sum\limits_{i=2}\limits^{n}\frac{2}{r_{i}+2}=\frac{p}{q-p}.$$ Then $$\mbox{d}_{G}(e_{1})=\frac{p}{q}.$$
\end{thm}
\begin{proof}
We first compute $\tau(G)$. For each $H_{r_{i}}$, $1\leq i\leq n$, it is not difficult to verify that
$$\tau(H_{r_{i}})=\tau(H_{r_{i}}-e_i)+\tau(H_{r_{i}}/e_i)=r_i2^{r_i-1}+2^{r_i}=2^{r_i-1}(r_i+2),$$
where $\tau(H_{r_{i}}/e_i)$ represents the graph obtained from $H_{r_i}$ by contracting $e_i$.
Hence
$$d_{H_{r_{i}}}(u_{i}v_{i})=\frac{\tau_{H_{r_{i}}}(u_{i}v_{i})}{\tau(H_{r_{i}})}=\frac{\tau(H_{r_{i}}/e_i)}{\tau(H_{r_{i}})} =\frac{2^{r_{i}}}{2^{r_{i}-1}(2+r_{i})}=\frac{2}{2+r_{i}}.$$
Then by Lemma~\ref{lemma3}, we get
\begin{align*}
\tau(G)&=\prod\limits_{i=1}\limits^{n}\tau(G_{i})\sum\limits_{i=1}\limits^{n}d_{G_{i}}(u_{i}v_{i})=\prod\limits_{i=1}\limits^{n}2^{r_{i}-1}(2+r_{i})\sum\limits_{i=1}\limits^{n}\frac{2}{2+r_{i}},\\
\tau_{G}(e_{1})&=\prod\limits_{i=1}\limits^{n}\tau(G_{i})\sum\limits_{i=2}\limits^{n}d_{G_{i}}(u_{i}v_{i})=\prod\limits_{i=1}\limits^{n}2^{r_{i}-1}(2+r_{i})\sum\limits_{i=2}\limits^{n}\frac{2}{2+r_{i}}.
\end{align*}
Thus
 \begin{equation*}
d_{G}(e_{1})=\frac{\tau_{G}(e_{1})}{\tau(G)}=\frac{\sum\limits_{i=2}\limits^{n}\frac{2}{2+r_{i}}}{\sum\limits_{i=1}\limits^{n}\frac{2}{2+r_{i}}}
=\frac{\sum\limits_{i=2}\limits^{n}\frac{2}{2+r_{i}}}{1+\sum\limits_{i=2}\limits^{n}\frac{2}{2+r_{i}}}=\frac{\frac{p}{q-p}}{1+\frac{p}{q-p}}=\frac{p}{q}.
 \end{equation*}
\end{proof}

 Then we will give the spanning tree edge number of the non-key edges in the graph $G=N(H_{0},H_{r_{2}},\cdots, H_{r_{n}})$.
\begin{thm}\label{t{G}(xy)}
Let $G=N(H_{0},H_{r_{2}},\cdots, H_{r_{n}})$ and $r_{i}\geq1$, for $i\geq2$. If $x_{k}y_{k}\in E(H_{r_{k}})$ is non-key edge of $G$, then
\begin{equation}
\tau_{G}(x_{k}y_{k})=\prod\limits_{i=1}\limits^{n}2^{r_{i}-1}(2+r_{i})\left[\frac{r_{k}+3}{2(2+r_{k})}\sum\limits_{i=1}\limits^{n}\frac{2}{2+r_{i}}-\frac{1}{(2+r_{k})^{2}}\right].
\end{equation}
\end{thm}
\begin{proof}
To compute $\tau_{H_{r_{i}}}(x_{i}y_{i})$, according to the structure of $H_{r_{i}}$, we can know
$$\tau_{H_{r_{i}}}(x_{i}y_{i})=\tau_{H_{r_{i}}-e_{i}}(x_{i}y_{i})+\tau_{H_{r_{i}}\setminus e_{i}}(x_{i}y_{i})=2^{r_{i}-1}+(r_{i}-1)2^{r_{i}-2}+2^{r_{i}-1}=(r_{i}+3)2^{r_{i}-2},$$
and $\tau_{H_{i}}(x_{i}y_{i};u_{i}v_{i})=\tau_{H_{r_{i}}-e_{i}}(x_{i}y_{i})=2^{r_{i}-1}$, so $d_{H_{r_{i}}}(x_{i}y_{i})=\frac{\tau_{H_{r_{i}}}(x_{i}y_{i})}{\tau(H_{r_{i}})}=\frac{(r_{i}+3)2^{r_{i}-2}}{2^{r_{i}-1}(2+r_{i})}=\frac{r_{i}+3}{2(2+r_{i})}$, for any $x_{i}y_{i}\in E(H_{r_{i}})$ is non-key edge of $G$. And we known $b_{G_{i}}(x_{i}y_{i};\ u_{i},v_{i})=\tau_{H_{i}}(x_{i}y_{i};u_{i}v_{i})$ combine with Lemma~\ref{lemma3}, for $x_{k}y_{k}\in E(R_{r_{k}})$
\begin{align*}
\tau_{G}(x_{k}y_{k})&=\prod\limits_{i=1}\limits^{n}\tau(G_{i})\left[\frac{b_{G_{k}}(x_{k}y_{k};\ u_{k},\
  v_{k})}{\tau(G_{k})}+d_{G_{k}}(x_{k}y_{k})\sum\limits_{i\neq k}d_{G_{i}}(u_{i}v_{i})\right]\\
  &=\prod\limits_{i=1}\limits^{n}2^{r_{i}-1}(2+r_{i})\left[\frac{2^{r_{k}-1}}{2^{r_{k}-1}(2+r_{k})}+\frac{r_{k}+3}{2(2+r_{k})}\sum\limits_{i\neq k}\limits^{n}\frac{2}{2+r_{i}}\right]\\
  &=\prod\limits_{i=1}\limits^{n}2^{r_{i}-1}(2+r_{i})\left[\frac{1}{2+r_{k}}-\frac{r_{k}+3}{2(2+r_{k})}\cdot\frac{2}{2+r_{k}}+\frac{r_{k}+3}{2(2+r_{k})}\sum\limits_{i=1}\limits^{n}\frac{2}{2+r_{i}}\right]\\
  &=\prod\limits_{i=1}\limits^{n}2^{r_{i}-1}(2+r_{i})\left[\frac{r_{k}+3}{2(2+r_{k})}\sum\limits_{i=1}\limits^{n}\frac{2}{2+r_{i}}-\frac{1}{(2+r_{k})^{2}}\right]
\end{align*}
\end{proof}

We now construct the the spanning tree edge dependence of $G=N(H_{0},H_{r_{2}},\cdots, H_{r_{n}})$.

\begin{thm}\label{tm3}
Let $p,\ q$ be positive integers, $\frac{q}{2}<p< q$. Let $G=N(H_{0},H_{r_{2}},\cdots, H_{r_{n}})$ such that
 $$\sum\limits_{i=2}\limits^{n}\frac{2}{r_{i}+2}=\frac{p}{q-p}\quad \mbox{and} \quad r_{i}\geq \frac{4q-6p}{2p-q}$$ for all $2\leq i\leq n$. Then $\mbox{dep}(G)=\frac{p}{q}$.
\end{thm}

\begin{proof}
 The above conditions $p< q$ and $\sum\limits_{i=2}\limits^{n}\frac{2}{r_{i}+2}=\frac{p}{q-p}$ satisfy Theorem~\ref{d{G}(e)}, we have $d_{G}(e_{1})=\frac{p}{q}$. According to Lemma~\ref{lemma5}, we can get $d_{G}(e_{1})=1\geq d_{G}(e_{k})$, for any key edge $e_{k}$, with $2\leq k\leq n$. According to Theorem~\ref{d{G}(e)} and Theorem~\ref{t{G}(xy)}, for any non-key edge $x_{k}y_{k}\in E(H_{r_{k}})$,
\begin{align*}
&\tau_{G}(e_{1})-\tau_{G}(x_{k}y_{k})\\
&=\prod\limits_{i=1}\limits^{n}2^{r_{i}-1}(2+r_{i})\sum\limits_{i=2}\limits^{n}\frac{2}{2+r_{i}}-\prod\limits_{i=1}\limits^{n}2^{r_{i}-1}(2+r_{i})
\left[\frac{r_{k}+3}{2(2+r_{k})}\sum\limits_{i=1}\limits^{n}\frac{2}{2+r_{i}}-\frac{1}{(2+r_{k})^{2}}\right]\\
&=\prod\limits_{i=1}\limits^{n}2^{r_{i}-1}(2+r_{i})\left[\sum\limits_{i=2}\limits^{n}\frac{2}{2+r_{i}}-\frac{r_{k}+3}{2(2+r_{k})}\sum\limits_{i=1}\limits^{n}\frac{2}{2+r_{i}}+\frac{1}{(2+r_{k})^{2}}\right]\\
&=\prod\limits_{i=1}\limits^{n}2^{r_{i}-1}(2+r_{i})\left[\left(1-\frac{r_{k}+3}{2(2+r_{k})}\right)\sum\limits_{i=2}\limits^{n}\frac{2}{2+r_{i}}-\frac{r_{k}+3}{2(2+r_{k})}+\frac{1}{(2+r_{k})^{2}}\right]\\
&=\prod\limits_{i=1}\limits^{n}2^{r_{i}-1}(2+r_{i})\left[\frac{r_{k}+1}{2(2+r_{k})}\sum\limits_{i=2}\limits^{n}\frac{2}{2+r_{i}}-\frac{r_{k}+3}{2(2+r_{k})}+\frac{1}{(2+r_{k})^{2}}\right]\\
&=\prod\limits_{i=1}\limits^{n}2^{r_{i}-1}(2+r_{i})\left[\frac{r_{k}+1}{2(2+r_{k})}\frac{p}{q-p}-\frac{r_{k}+3}{2(2+r_{k})}+\frac{1}{(2+r_{k})^{2}}\right]\\
&=\prod\limits_{i=1}\limits^{n}2^{r_{i}-1}(2+r_{i})\left[\frac{(r_{k}+1)(2+r_k)}{2(2+r_{k})^2}\frac{p}{q-p}-\frac{(r_{k}+3)(2+r_k)-2}{2(2+r_{k})^2}\right]\\
&=\prod\limits_{i=1}\limits^{n}2^{r_{i}-1}(2+r_{i})\frac{r_k+1}{2(2+r_{k})}\left[(2+r_k)\frac{p}{q-p}-(r_{k}+4)\right]\\
&=\prod\limits_{i=1}\limits^{n}2^{r_{i}-1}(2+r_{i})\frac{r_k+1}{2(2+r_{k})}\left(\frac{2p-q}{q-p}r_{k}-\frac{4q-6p}{q-p}\right)\\
&\geq\prod\limits_{i=1}\limits^{n}2^{r_{i}-1}(2+r_{i})\frac{r_k+1}{2(2+r_{k})}\left(\frac{2p-q}{q-p}\cdot\frac{4q-6p}{2p-q}-\frac{4q-6p}{q-p}\right)\\
&=0.
\end{align*}
Since for any non-key edge $x_{k}y_{k}\in E(H_{r_{k}})$, $0\leq k\leq n$, we obtain that $\tau_{G}(e_{1})>\tau_{G}(x_{k}y_{k})$. Then $d_{G}(e_{1})>d_{G}(x_{k}y_{k})$, so $\mbox{dep}(G)=max_{e\in E(G)}d_{G}(e)=d_{G}(e_{1})=\frac{p}{q}$.
\end{proof}

\begin{cor}\label{cor}
For any rational number $\frac{1}{2}<\frac{p}{q}<1$, $p$, $q$ are positive integers, there exists a planar graph $G$, such that $\mbox{dep}(G) = \frac{p}{q}$.
\end{cor}

\section{Concluding remarks}
In this paper, we show that all rational spanning tree edge dependences are constructible for bipartite graphs, which completely solve the first conjecture of Kahl. For the second conjecture of Kahl, we show that the conjecture is true for planar multigraphs. However, for (simple) planar graphs, the conjecture is not true. We show that for any rational number $p/q$ such that $p/q\leq 1/3$, the dependence $p/q$ is not constructible via planar graphs. On the other hand, for $1/2<p/q$, we show that the dependence $p/q$ is constructible via planar graphs. Thus, for $1/3<p/q\leq 1/2$, it still remains an open question whether the dependence $p/q$ is constructible with planar graphs. It deserves further discussing and studying in the future. What is more, it is of special interest to determine the minimum rational number $p/q$ such that $p/q$ is constructible via planar graphs. So we propose the following question.

{\bf Question.} Which is the minimum rational number $0<p/q<1$ such that $p/q$ is constructible via planar graphs?

\section{Ackonwledgements}
The support of the National Natural Science
Foundation of China (through grant no. 12171414) and the project ZR2019YQ02 by the Shandong
Provincial Natural Science Foundation, is greatly acknowledged.

\bibliographystyle{plain}

\begin{thebibliography}{9}
{\small

\bibitem{nka} N. Kahl, \emph{On constructing rational spanning tree edge densities}, Discrete Appl. Math. \textbf{213} (2016), 224-232.

\bibitem{kir} G. Kirchhoff, \emph{Uber die aufl\"{o}sung der gleichungen, auf welche man bei der untersuchung der linearen verteilung galvanischer str\"{o}me gef\"{u}hrt wird Ann}, Annalen der Physik. \textbf{148} (1847), 497-508

\bibitem{cay} A. Cayley, \emph{A theorem on trees}, Quart. J. Math. \textbf{23} (1889), 376-378.

\bibitem{sco} H. I. Scoins, \emph{The number of trees with nodes of alternate parity}, Proc. Camb. Philos. Soc. \textbf{58} (1962), 12-16.

\bibitem{sbe} M. Abu-Sbeih, \emph{On the number of spanning trees of $K_{n}$ and $K_{m,n}$}, Discrete Math. \textbf{84} (1990), 205-207.

\bibitem{aus} T. L. Austin, \emph{The enumeration of point labelled chromatic graphs and trees}, Can. J. Math. \textbf{12} (1960), 535-545.

\bibitem{lew} R. P. Lewis, \emph{The number of spanning trees of a complete multipartite graph}, Discrete Math. \textbf{197} (1999), 537-541.

\bibitem{neil} P. O'Neil, \emph{Enumeration of spanning trees in certain graphs}, IEEE Trans. Circuit Theory. \textbf{17} (1970), 250.

\bibitem{yta} X. R. Yong and T. Acenjian, \emph{The numbers of spanning trees of the cubic cycle $C_N^3$ and the quadruple cycle $C_N^4$}, Discrete Math. \textbf{169} (1997), 293-298.

\bibitem{nr} S. D. Nikolopoulos and P. Rondogiannis,\emph{ On the number of spanning trees of multi-star related graphs}, Inform. Process. Lett. \textbf{65} (1998), 183-188.

\bibitem{nr1} S. D. Nikolopoulos and C. Papadopoulos, \emph{The number of spanning trees in $K_n$-complements of quasi-threshold graphs}, Graphs and Combinatorics. \textbf{20} (2004), 383-397.

\bibitem{nr2} S. D. Nikolopoulos and C. Papadopoulos, \emph{On the number of spanning trees of $K_n^m\pm G$ graphs}, Discrete Mathematics and Theoretical Computer Science. \textbf{8} (2006), 235-248.

\bibitem{zyg} Y. Zhang, X. Yong, and M. J. Golin, \emph{The number of spanning trees in circulant graphs}, Discrete Math. \textbf{223} (2000), 337-350.

\bibitem{lpw} Z. Lonc, K. Parol, and J. M. Wojciechowski, \emph{On the number of spanning trees in directed circulant graphs}, Networks. \textbf{37} (2001), 129--133.

\bibitem{clz} X. Chen, Q. Lin, and F. Zhang, \emph{The number of spanning trees in odd valent circulant graphs}, Discrete Math. \textbf{282} (2004), 69-79.

\bibitem{gyy} Y. Zhang, X. Yong, and M. J. Golin, \emph{Chebyshev polynomials and spanning tree formulas for circulant and related graphs}, Discrete Math. \textbf{298} (2005), 334--364.


\bibitem{gyz} M. J. Golin, X. Yong, and Y. Zhang, \emph{The asymptotic number of spanning trees in circulant graphs}, Discrete Math. \textbf{310} (2010), 792-803.

\bibitem{lcry} M. Li, Z. Chen, X. Ruan, and X. Yong, \emph{The formulas for the number of spanning trees in circulant graphs}, Discrete Math. \textbf{338} (2015), 1883-1906.

\bibitem{npp} S. D. Nikolopoulos, L. Palios, and C. Papadopoulos, \emph{ Maximizing the number of spanning trees in $K_n$-complements of asteroidal graphs}, Discrete Math.  \textbf{309} (2019), no. 10, 3049-3060.

\bibitem{yz} W. Yan and F. Zhang, \emph{Enumeration of spanning trees of graphs with rotational symmetry}. Journal of Combinatorial Theory. \textbf{118} (2011), no. 4, 1270-1290.

\bibitem{yan} W. Yan, \emph{On the number of spanning trees of some irregular line graphs}, J. Combin. Theory Ser. A. \textbf{120} (2013), 1642-1648.

\bibitem{dy} F. Dong and W. Yan, \emph{Expression for the Number of Spanning Trees of Line Graphs of Arbitrary Connected Graphs}, J. Graph Theory. \textbf{85} (2017), 74-93.

\bibitem{gj} H. Gong and X. Jin, \emph{A simple formula for the number of spanning trees of line graphs}, J. Graph Theory. \textbf{88} (2018), 294-301.

\bibitem{my} F. Ma and B. Yao, \emph{The number of spanning trees of a class of self-similar fractal models}, Inform. Proc. Lett. \textbf{136} (2018), 64-69.

\bibitem{ly} T. Li and W. Yan, \emph{Enumeration of spanning trees of 2-separable networks}, Physica A: Statistical Mechanics and its Applications. \textbf{536} (2019), 120877.

\bibitem{zyan} J. Zhang and W. Yan, \emph{Counting spanning trees of a type of generalized Farey graphs}, Physica A: Statistical Mechanics and its Applications. \textbf{555} (2020), 124749.

\bibitem{gd} J. Ge and F. Dong, \emph{Spanning trees in complete bipartite graphs and resistance distance in nearly complete bipartite graphs}, Discrete Applied Mathematics. \textbf{283} (2020), 542-554.

\bibitem{moy} M. Haruhide, K. Ozeki, and T. Yamashita, \emph{Spanning trees with a bounded number of branch vertices in a claw-free graph}, Graphs and Combinatorics. \textbf{30} (2014), no. 2, 429-437.

\bibitem{fgs} M. Ferrara, R. Gould, and C. Suffel, \emph{Spanning tree edge densities}, in: Proceedings of the Thirty-third Southeastern International Conference on Combinatorics, Graph Theory and Computing, Cong. Numer., \textbf{154} (2002), 155-163.

\bibitem{kr} D. J. Klein and M. Randic, \emph{Resistance distance}, J. Math. Chem. \textbf{12} (1993), no. 1, 81-95.

\bibitem{fs} M. Fiedler and J. Sedl\'{a}\v{c}ek, \emph{\"{U} ber Wurzelbasen von gerichteten graphen}, Casopis Pest. Mat. \textbf{83} (1958), 214-225.

\bibitem{bbo} B. B$\acute{e}$la and B. Bollobas, \emph{Modern graph theory}, Springer Science \& Business Media. \textbf{184} (1998).

\bibitem{tho} C. Thomassen, \emph{Resistances and currents in infinite electrical networks}, J. Combina. Theory Ser. B. \textbf{49} (1990), 87-102

\bibitem{foster} R.M. Foster, \emph{The average impedance of an electrical network}, In Contributions to Applied Mechanics (Reissner Anniversary Volume), 1949, 333-340.


}
\end{thebibliography}

\end{document}